\documentclass[12pt]{amsart}
\usepackage{graphicx}

\usepackage{amssymb}

\newtheorem{theorem}{Theorem}[section]
\newtheorem{lemma}[theorem]{Lemma}
\newtheorem{corollary}[theorem]{Corollary}
\newtheorem{proposition}[theorem]{Proposition}

\DeclareMathOperator{\sgn}{sgn}
\begin{document}

\title[Few Long Lists] {Few Long Lists for Edge Choosability\\ of Planar Cubic Graphs}

\author{Luis Goddyn}
\address{Department of Mathematics, Simon Fraser University, Burnaby, BC, Canada}
\email { goddyn@sfu.ca}
 
\author{Andrea Spencer}
\address{Department of Mathematics, Simon Fraser University, Burnaby, BC, Canada}
\email { ams33@sfu.ca}

\thanks{Supported by a Canada NSERC Discovery Grant.}

\keywords{edge list coloring, edge choosbility, four color theorem, polynomial method, Nullstellensatz}

\date{October 20, 2012}

\begin{abstract}
It is known that every loopless cubic graph is $4$-edge choosable.
We prove the following strengthened result.

Let $G$ be a planar cubic graph having $b$ cut-edges.
There exists a set $F$ of at most $\frac52b$ edges of $G$ with the following property.
For any function $L$ which assigns to each edge of $F$ a set of $4$ colours
and which assigns to each edge in $E(G)-F$ a set of $3$ colours, the graph $G$ has a proper edge colouring
where the colour of each edge $e$ belongs to $L(e)$.
\end{abstract}

\maketitle

\section{Introduction}

We assume here that graphs are finite and loopless.
An \emph{edge list assignment} for a graph $G$ is a function $L$ which maps each edge of $G$
a set of colours.   An \emph{$L$-edge colouring} is a proper edge colouring $c$ of $G$ for which $c(e) \in L(e)$ for each $e \in E(G)$.
Let $f : E(G) \to \mathbb N$ be an edge weighting with positive integers. 
We say that $G$ is \emph{$f$-edge choosable} if $G$ has an $L$-edge colouring for every edge list assignment $L$ satisfying $|L(e)| \ge  f(e)$ for each $e\in E(G)$.
For $k \in \mathbb N$, we say that $G$ is \emph{$k$-edge choosable} if $G$ is $f$-edge choosable for some $f$ satisfying $\max f \le k$.
If, additionally, $G$ has a most $s$ edges $e$ for which $f(e)= k$, then we say that $G$ is 
\emph{$s$-nearly $(k-1)$-edge choosable}.
We consider the problem finding a good upper bound on the quantity
\[
s(G,k) := \min \{ s \;\vert\;  \text{ $G$ is $s$-nearly $k$-edge choosable}\}.
\]

The notion of \emph{vertex choosability} is defined similarly, with reference to vertex colouring. 
An extension \cite{ERT} of Brooks' theorem to vertex choosability asserts that
every simple connected graph $H$ with maximum degree~$\Delta$ is $\Delta$-vertex choosable unless $H$ is a complete graph or a cycle.
Applying this to the line-graph of a cubic graph, it follows that every loopless cubic graph is $4$-edge choosable.
However, a typical cubic graph $G$ is 
$s$-nearly $3$-edge choosable where $s$ is somewhat smaller than the order of $G$.
For example, a cubic graph $G$ satisfies $s(G,3)=0$ (that is, $G$ is $3$-edge choosable) if $G$ is either bipartite \cite{Galvin}, or planar and 2-connected \cite{Jaeger}.
The latter result strengthens the 4-colour theorem.
If $G$ has a cut-edge, then $s(G,3)>0$ since $G$ is not 3-edge colourable.
One easily constructs cubic graphs $G$ having $b$ cut-edges for which $s(G,3) \ge 2b$.
For example, if each connected component of $G$ has exactly one cut-edge, 
then $G$ is $f$-edge choosable only if each of the $2b$ leaf-blocks of $G$ contains at least one edge $e$ for which $f(e) \ge 4$.
(This is because no leaf block is $3$-edge colourable).
In this paper we show that planar cubic graphs $G$ satisfy $s(G,3) \le \frac52b$.

\begin{theorem}
\label{thm:main}
Let $G$ be a planar cubic graph having $b$ cut-edges.
Then $G$ is $f$-edge choosable for some function $f:E(G) \to \{1,2,3,4\}$ that has 
average value $3$, and $|f^{-1}(4)| \le \frac52b$.
\end{theorem}

\section{The Polynomial Method}
Let $e_1, e_2, \dots , e_m$ be the edges of a graph $G$, and let $x_i$ be an indeterminate associated with the edge $e_i$.
The \emph{edge monomial} of $G$ is the polynomial in $\mathbb R[x_1,\dots, x_m]$ defined by
\begin{equation}
\label{eq:monomial}
\epsilon(G) = \prod_{1 \le i < j \le m} (x_i-x_j)^{c(i,j)} .
\end{equation}
Here $c(i,j) \in \{0,1,2\}$ is the number of vertices incident to both $e_i$ and $e_j$.
Note that $\epsilon(G)$ is a homogeneous polynomial of degree $e(G)$.
Furthermore, $\epsilon(G)$ is well defined (up to negation) regardless of the edge ordering $e_1,\dots, e_m$.
Each term in the standard expansion of $\epsilon(G)$ takes the form $\alpha_w x^w := \alpha_w \prod_{1 \le i \le m} x_i^{w(e_i)}$
where the exponent function $w:E(G) \to \{0,1,\dots\}$ is a nonnegative integer weighting of the edges of $G$.
We shall write $w+\mathbf 1$ for the function $e \mapsto w(e)+1$.
The Combinatorial Nullstellensatz \cite{Alo, AloTar} for edge choosability asserts the following.
\begin{lemma}
\label{lem:CN}
Let $G$ be a loopless graph, and let $\alpha_w x^w$ be a nonzero term in the expansion of~$\epsilon(G)$.
Then $G$ is $f$-edge choosable, where $f=w+\mathbf 1$.
\end{lemma}
The \textit{polynomial method} for proving that $G$ is $f$-edge choosable typically involves
selecting an exponent function $w$ satisfying $w+ \mathbf 1 \le f$, and showing that $\alpha_w \ne 0$.
To evaluate $\alpha_w$, Ellingham and Goddyn \cite{ElGo} provide a
combinatorial interpretation of $\alpha_w$ in terms of \emph{star  labellings} which we describe below.
Let $v$ be a vertex of degree $d$ in $G$.
A \emph{star labelling at~$v$} is a bijective function $\pi_v$ from the edges incident with $v$ to the integers $\{0,1,\dots,d-1\}$.
A \emph{star labelling} of $G$ is a set $\pi = \{\pi_v : v \in V(G)\}$ where each $\pi_v$ is a star  labelling at~$v$.
The \emph{exponent} of a star labelling $\pi$ is the edge weighting $w=w_\pi$ defined by $w(e_i) = \pi_u(e_i) + \pi_v(e_i)$, for $e_i=uv\in E(G)$.
The \emph{sign}, $\sgn (\pi_v)$, of a star  labelling at $v$ is the sign of the permutation of $\{0,1,\dots , d-1\}$ defined by
$j \mapsto \pi_v( e_{i_j})$, where $e_{i_0}, e_{i_1}, \dots, e_{i_{d-1}}$ are the edges incident with $v$,
and $i_0 < i_1 < \dots < i_{d-1}$.
The \emph{sign} of a star labelling of $G$ is defined by $\sgn (\pi) = \prod_{v\in V(G)} \sgn(\pi_v)$.
\begin{lemma}[\cite{ElGo}]
\label{lem:combInterp}
For any loopless graph $G$ we have 
\begin{equation}
\label{eq:combInterp}
\epsilon(G) = \sum_\pi \sgn (\pi) \, x^{w_\pi},
\end{equation}
where the sum is taken over the set of star  labellings of $G$.
\end{lemma}
The reader should notice that, up to negation, the edge monomial $\epsilon(G)$ does not depend on the particular ordering $e_1, e_2,\dots,e_m$ of $E(G)$.
A novel feature of this paper is our use of several coefficients of $\epsilon(G)$ in the polynomial method.
If a set of coefficients $\{ \alpha_{w_i} \mid\; 1 \le i \le k\}$ has a nonzero sum, then at least one coefficient $\alpha_{w_i}$ is not zero.
This gives a multi-term version of Lemma \ref{lem:CN}.
\begin{corollary}
\label{cor:interp}
Let $\mathcal W = \{w_1, \dots w_k\}$ be a set of edge weightings for $G$, and let $\Pi (\mathcal W)$ be the set of star labellings $\pi$ of $G$
such that the exponent of $\pi$ is a member of $\mathcal W$.
If the integer 
\begin{equation}
\label{eq:Wsum}
\sum_{\pi\in \Pi(\mathcal W)} \sgn(\pi)
\end{equation}
 is not zero, 
 then $G$ is $(w_i+\mathbf 1)$-edge choosable, for some $i\in \{1,2,\dots,k\}$.
 \end{corollary}
\begin{proof}
Let $\Pi(w_i)$ be the set of star labellings of $G$ having exponent $w_i$.
If \eqref{eq:Wsum} is not zero, then 
$\sum_{\pi\in \Pi(w_i)} \sgn(\pi) \ne 0$, for some $i \in \{1,2,\dots k\}$.
By Lemma \ref{lem:combInterp} this last sum equals the coefficient
of the term $\alpha_{w_i} x^{w_i}$ in the expansion of $\epsilon(G)$,
and the result follows from Lemma~\ref{lem:CN}.
\end{proof}

To prove our main result, we will construct an appropriate set $\mathcal W$ of edge weightings of a planar cubic graph $G$,
and show that the signed sum \eqref{eq:Wsum} is positive.

\section{Weightings and Star  labellings of Threads}
\label{sec:threads}
A \emph{flag} of a graph $G$ is a pair $(v, e)\in V(G) \times E(G)$ whose members are incident in $G$.
We may write $ve$ instead of $(v,e)$.
It is convenient to regard a star  labelling of $G$
to be a nonnegative integer function $\pi : F(G) \to \{0,1,\dots, \Delta(G)-1\}$,
where $F(G)$ is the set of flags of $G$.
For $m\ge 0$, a \emph{thread of order $m$} is the graph $T_m$ obtained from a path $v_0e_0v_1e_1 \dots e_{m} v_{m+1}$
by adding new vertices $w_k$
and new edges $f_k = v_k w_k$ ($1 \le k \le m$).
See Figure \ref{fig:threadWeights}.
In particular the \emph{trivial thread} $T_0$ is the path $v_0 e_0 v_1$.
The \emph{head} of $T_m$ is  the flag $v_0 e_0$,
the \emph{tail} of $T_m$ is the flag $v_{m+1}e_{m}$, and
the \emph{feet} of $T_m$ are the flags $w_k f_k$ ($1 \le k \le m$).
A function $\pi : F(T_m) \to \{0,1,2\}$ is called a \emph{prestar  labelling of $T_m$}
if the restricted function $\pi_{v_k} := \pi \upharpoonright_{\{v_k e_{k-1}, v_k e_{k}, v_k f_k\}}$
 is a star  labelling of $v_k$,  for $1 \le k \le m$.
 The \emph{sign} of  a prestar labelling $\pi$ is defined to be $\sgn(\pi)=\prod_{k=1}^m \sgn(\pi_{v_k})$,
and the \emph{exponent} of $\pi$ is the edge weighting $w$ 
where $w(e) = \pi(ue) + \pi(ve)$ for each  $e=uv \in E(T_m)$.
A prestar labelling $\pi$ is \emph{1-footed} if $\pi(w_k f_k)=1$, for $1 \le k \le m$.
We say that  $\pi$ has \emph{type $(i,j)$} if $\pi(v_0 e_0)=i$ and $\pi(v_{m+1} e_{m})=j$.
We are interested in classifying, according to type,  the set of $1$-footed prestar labellings of $T_m$
which have a prespecified exponent.

For each $m\ge0$ we define four special edge weightings of $T_m$, which we denote by $w_{\mathbf 2}$, $w_{11}$, $w_{02}$ and $w_{20}$.
These are illustrated in Figure \ref{fig:threadWeights}.
The weighting $w_{\mathbf 2}$ is just the constant function $w(e) \equiv 2$.
The next three weightings are defined only for $m \ge 1$.
The weighting $w_{11}$ is obtained from  $w_{\mathbf 2}$
by \emph{transferring one unit of weight} from $e_{0}$ to $e_m$.
That is, we have  $w_{11}(e_{0})=1$,  $w_{11}(e_m)=3$,  and $w_{11}(e)=2$ for $e  \in E(T_m)-\{e_0,e_{m}\}$.
The weighting $w_{02}$ is obtained from $w_{\mathbf 2}$
by transferring one unit of weight from $e_0$ to $f_1$.
The weighting $ w_{20}$ is obtained from $w_{\mathbf 2}$
by transferring one unit of weight from $e_{1}$ to $e_0$,
and then transferring one unit of weight from $e_{1}$ to $f_1$.

As shown in Figure \ref{fig:threadWeights}, each special weighting is associated with one or more $1$-footed prestar labellings of $T_m$.
Each labelling is denoted by either $\rho_{ij}$, $\pi_{ij}$ or $\pi'_{ij}$ where $(i,j)$ is its type.

\begin{figure}
\hspace{-63mm}\includegraphics[width=103mm]{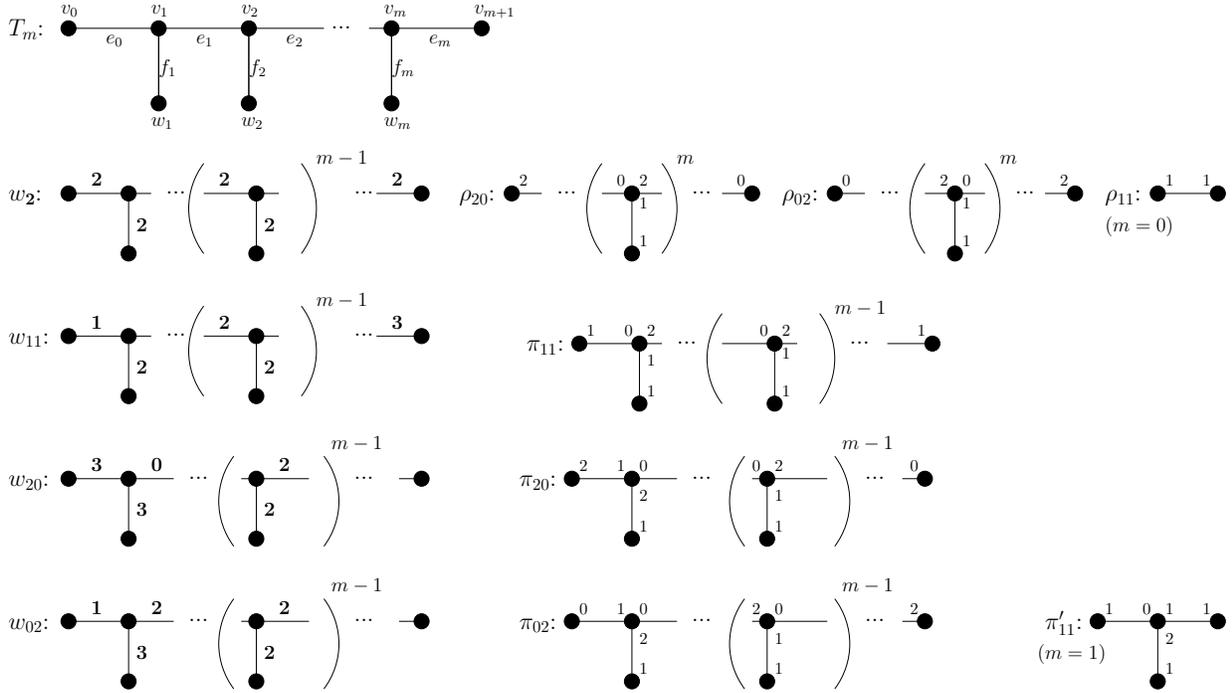}
\caption{The thread $T_m$,  four edge weightings $w_{\mathbf 2}$ and $w_{ij}$, and all $1$-footed prestar labellings, $\rho_{ij}$, $\pi_{ij}$ and $\pi'_{11}$, whose exponents are one of those weights.}
\label{fig:threadWeights}
\end{figure}

\begin{lemma}    \label{lem:possibleStars}
There are three prestar labellings $\rho_{20}$, $\rho_{02}$ and $\rho_{11}$ of $T_0$ with exponent $w_{\mathbf 2}$.
For $m\ge 1$,  $\rho_{20}$ and $\rho_{02}$ are the only $1$-footed prestar labellings of $T_m$ having exponent $w_{\mathbf 2}$.
Let $m\ge1$ and let $(i,j) \in \{(1,1), (2,0), (0,2)\}$.
If $(m,i,j) \ne (1,0,2)$, then $\pi_{ij}$ is the unique $1$-footed prestar labelling of $T_m$ having exponent $w_{ij}$.
If $(m,i,j)=(1,0,2)$, then $T_1$ has exactly two 1-footed prestar labellings with exponent $w_{02}$, namely $\pi_{02}$ and $\pi'_{11}$.
\end{lemma}

\begin{proof}
We prove only the statement regarding $w_{02}$ since the arguments are easy and mechanical.
Let $\pi$ be a 1-footed prestar labelling of $T_m$ whose exponent equals $w_{02}$.
Since $\pi$ is $1$-footed, 
we have $\pi(v_1 f_1) = w_{02}(f_1) - \pi(w_1 f_1) =3-1=2$,
and $\pi(v_k f_k) = 2- 1 =1$, for $2\le k \le m$.
Since $\pi_{v_1}$ is a star labelling we have $\{\pi(v_1 e_0), \pi(v_1 e_1)\}=\{0,1\}$.
In case $\pi(v_1 e_1) = 0$, we have $\pi(v_1 e_0) =1$ and $\pi(v_0 e_0) = 1-1=0$.
We now apply the facts $\pi(v_{k+1} e_k) = 2- \pi( v_k e_k)$ ($k = 1,2,\dots,m$), and
$\{\pi(v_{k} e_{k-1}), \pi(v_k e_k)\}=\{0,2\}$ ($k=2,3,\dots,m$) to find
$\pi(v_k e_k) =0$ and $\pi(v_{k+1} e_k)=2$ for $k=1,2,\dots,m$.
Thus $\pi = \pi_{02}$.
In case $\pi(v_1 e_1)=1$, we find that
$\pi(v_1 e_0) =0$, $\pi(v_0 e_0) = 1-0=1$ and $\pi(v_2 e_1)=2-1=1$. 
If $m \ge 2$, then we have contradicted the fact $\pi_{v_2}$ is a star labelling, since $\pi(v_2 f_2) =1$.
Therefore $m=1$, and $\pi$ is the exceptional prestar labelling $\pi'_{11}$.
\end{proof}

\begin{proposition}
\label{prop:threadSigns}
For each odd integer $m \ge 1$, the prestar  labellings of $T_m$ defined above satisfy $\sgn(\pi_{20}) = \sgn(\pi_{02})$.
For $m=1$, we have that $\sgn(\pi'_{11}) = -\sgn(\pi_{02})$.
For each even integer $m \ge 2$ we have that $\sgn(\rho_{02}) = \sgn(\rho_{20})$.
\end{proposition}
\begin{proof}
Consider the embedding of $T_m$ in the plane shown in Figure \ref{fig:threadWeights}.
The sign of a star  labelling at $v_k$ depends only on whether the three labels $0,1,2$ appear in clockwise or anticlockwise order around~$v_k$.
When $m=1$, the star  labelling at $v_1$ is clockwise under $\pi'_{11}$, and is
anticlockwise under~$\pi_{02}$.  Therefore $\sgn(\pi'_{11}) = -\sgn(\pi_{02})$ for any embedding of $T_1$.
Evidently, every vertex $v_k$ ($1 \le k \le m$) is anticlockwise under $\pi_{20}$,
whereas $v_1$ is the unique anticlockwise vertex under $\pi_{02}$.
Therefore $\sgn(\pi_{20}) = \sgn(\pi_{02})$  provided that $m$ is odd.
Each vertex $v_k$ is clockwise under $\rho_{02}$ and anticlockwise under $\rho_{20}$,
so $\sgn(\rho_{20}) = \sgn(\rho_{02})$ when $m$ is even.
\end{proof}

We define two variations of an edge-weighted thread.
For $m\ge1$, the \emph{closed thread of order~$m$} is the graph $T_m^\circ$ obtained from  $T_m$ by identifying the vertices $v_0$ and $v_{m+1}$. 
We define the edge weightings
  $w_{02}$,   $w_{\mathbf 2}$, 
 and the prestar labellings 
 $\pi_{02}$, $\rho_{20}$ and $\rho_{02}$ exactly as they were defined for $T_m$.
The reader will easily verify the following lemma.

\begin{lemma}
\label{lem:uniq-circ}
Let $m\ge 1$. 
Then  $\rho_{20}$ and $\rho_{02}$ are the only 1-footed prestar  labellings of $T_m^\circ$ having exponent~$w_{\mathbf 2}$.
Furthermore, $\pi_{02}$ is the unique 1-footed prestar  labelling of $T_m^\circ$ having exponent $w_{02}$, for which the head and tail (that is, the two flags incident with $v_0$) receive different labels.
\end{lemma}

An \emph{injured thread of order $m$} is any graph $T_m^-$ that is obtained from $T_m$ by deleting any one of its $m$ ``feet'' $w_k$, $1 \le k \le m$.
The edge weighting $w_{11}^-$ of $T_m^-$ is the restriction of $w_{11}$ to the edge set of $T_m^-$.
A \emph{$1$-footed prestar labelling} of $T_m^-$ is the restriction of a $1$-footed prestar labelling of $T_m$ to the flags in $T_m^-$.
In particular, we define $\pi_{11}^- = \pi_{11} \upharpoonright _{F(T_m^-)}$.
To ease notation, we shall write $w_{11}$ instead of $w_{11}^-$, and write $\pi_{11}$ instead of $\pi_{11}^-$, where no confusion results.
\begin{lemma}
\label{lem:uniq-deleted}
For $m\ge 1$, the prestar  labelling $\pi_{11}$ is the unique 1-footed prestar  labelling of $T_m^-$
whose exponent equals $w_{11}$.
\end{lemma}

\section{A Set of Edge Weightings}
\label{sec:weightSet}
In this section, we define a set of edge weightings $\mathcal W$ of a planar cubic graph, to which we will apply Corollary \ref{cor:interp}.
Let $G$ be a connected loopless cubic graph and let $B(G)$ be the set of cut-edges in $G$.
A \emph{block} of $G$ is any connected component of $G-B(G)$
(this differs  from the standard definition of ``block'').
Each block, $H$, is either a \emph{vertex block}, a \emph{cycle block} or a \emph{proper block},
depending on whether $H$ is a single vertex, a cycle or a subdivision of  $2$-connected cubic graph.
The \emph{block tree} of $G$ is the tree obtained by contracting each block $H$ to a single vertex, which we also denote by $H$ where no confusion results.
Since $G$ is finite, at least one block of $G$ is a proper block.
We designate one proper block to be the $\emph{root block}$ $H_0$ of $G$.
Every other block of $G$ is called a \emph{nonroot block} of $G$.
We define $B(H_0)$ to be the set of edges in $B(G)$ which have exactly one end in $H_0$.
Every nonroot block $H$ is incident to a unique cut-edge, denoted by $e_H$,
 which lies on the path from $H$ to ${H_0}$ in the block tree of $G$.
For nonroot blocks $H$, we define $B(H)$ to be the set of edges in $B(G)-\{e_H\}$ which have an endpoint in $H$.
For any block $H$ of $G$, let $H^+$ be the subgraph of $G$ obtained from $H$ by adding all the edges in $B(H)$ and their endpoints.
Each subgraph $H^+$ is called an \emph{extended block} of $G$.
The extended blocks of $G$ depend on the choice of $H_0$.

The edge sets of the extended blocks of $G$ form a partition of~$E(G)$.
 We further refine the extended blocks into pieces that are each isomorphic to one of the  threads, $T_m$, $T^\circ_m$ or $T^-_m$, defined in Section \ref{sec:threads}.
Each extended vertex block is a path of length $2$, which we regard to be copy of the injured thread $T^-_1$.
We define the family
\[
\mathcal T^1 = \{ H^+ \;\vert \text{ $H$ is a vertex block of $G$}\}.
\]
Each extended cycle block is isomorphic to a closed thread $T_m^\circ$, for some $m\ge 1$.
We group these into two families.
\begin{align*}
\mathcal T_\text{odd}^\circ &= \{ H^+ \;\vert \text{ $H$ is a cycle block of $G$, and $H^+ \cong T^\circ_m$, where $m$ is odd}\}\\
\mathcal T_\text{even}^\circ &= \{ H^+ \;\vert \text{ $H$ is a cycle block of $G$, and $H^+ \cong T^\circ_m$, where $m$ is even}\}.
\end{align*}
The reader should notice that if $H^+ \cong T^\circ_m$, then
the length of the cycle $H$ has opposite parity to $m$.%

Each extended proper block $H^+$ of $G$ decomposes into copies of 
threads $T_m$ and injured threads $T^-_m$ as follows.
By suppressing every vertex of degree $2$ in $H$ we obtain a  $2$-connected cubic graph homeomorphic to $H$, which is denoted $\bar H$ and called the \emph{derived graph} of $H$.
Each edge $\bar e \in E(\bar H)$  corresponds to a maximal induced path $P$ of positive length in $H$.
By adding to $P$ those edges in $B(H)$ (and their endpoints) which are incident to $P$,
we obtain a subgraph $T_{\bar e} \subseteq H^+$ which is isomorphic to either
a thread $T_m$, $m\ge0$, or injured thread $T^-_m$, $m\ge 1$.   
The  edge sets of the subraphs in  $\{T_{\bar e} \;\mid\; \bar e \in E(\bar H)\}$ form a partition of $E(H^+)$.  
%
Summarizing, we have decomposed $G$ into a family
$\mathcal T =  \mathcal T^1 \cup \mathcal T^\circ_\text{odd} \cup  \mathcal T^\circ_\text{even}  \cup \{T_{\bar e}\;\vert\; \bar e \in E(\bar G)\}$ 
of copies of threads, injured threads and closed threads where
\begin{align*}
\bar G   &= \cup \{ \bar H \;\vert \text{ $H$ is a proper block of $G$}\}.
\end{align*}
The members of $\mathcal T$ are called \emph{general threads of $G$},
and $\bar G$ is the \emph{derived graph} of $G$.  
An edge $\bar e \in E(\bar G)$ is a \emph{base edge} of $\bar G$ if $T_{\bar e}$ is isomorphic to an injured thread.  Thus each connected component of $\bar G$ other than $\bar{H_0}$ contains exactly one base edge.
Each $T\in \mathcal T$ has zero or more well defined \emph{feet},
but there are two ways to select which end is the  \emph{head}  of $T$.

\smallskip

Let $G$ be a connected planar cubic graph where a base block $H_0$ has been selected.
Let the $\mathcal T$ and $\bar G$ be the general thread decomposition and reduced graph as defined above.
To describe a weighting of $G$ it suffices to specify, for each  $T \in \mathcal T$,
which end of $T$ is the head, and which of the weightings described in Section \ref{sec:threads} to assign to $T$.  
This specification will make reference to a particular perfect matching in the reduced cubic graph $\bar G$.
If $M\subseteq E(\bar G)$ is a perfect matching in $\bar G$, then the edge set
$D=E(\bar G)-M$ is a \emph{$2$-factor} of $\bar G$.
We say that $D$ is \emph{bipartite} if every cycle of $\bar G-M$ has even length.
Let $M \subseteq E(\bar G)$ be a perfect matching in $\bar G$ satisfying the following properties.
\begin{enumerate}
\label{eq:MProps}
\item every base edge of $\bar G$ is an edge in $M$, \label{M:1}
\item the $2$-factor $D= E(\bar G) - M$ is bipartite,
\item subject to conditions (1) and (2), $M$ contains the maximum possible number of  edges~$\bar e$ for which the general thread $T_{\bar e}$ is nontrivial (that is, $T_{\bar e} \not\cong T_0$).
\end{enumerate}
The matching $M$ exists because every component of $\bar G$ has a proper $3$-edge colouring
(by the Four Colour Theorem), and has at most one base edge.
We partition the set of threads
$\{T_{\bar e} \;\vert\; \bar e \in E(\bar G) \}$ into five classes 
$(\mathcal T_0^M, \mathcal T_{\ge 1}^M,\mathcal T^D_{\textrm{odd}}, \mathcal T^D_{\textrm{even}},\mathcal T^D_0)$ where
\begin{align*}
\mathcal T_0^M &= \{ T_{\bar e} \in \mathcal T \;\vert\; e_T \in M, T_{\bar e} \cong T_0 \text{ is trivial}\},\\
\mathcal T_{\ge 1}^M &= \{ T_{\bar e} \in \mathcal T \;\vert\; e_T \in M, T_{\bar e}  \text{ is nontrivial}\},\\
\mathcal T^D_{\textrm{odd}} &= \{ T_{\bar e} \in \mathcal T \;\vert\; e_T \in D,\text{ and $T_{\bar e}$ has odd order}\}\\
\mathcal T^D_{\textrm{even}} &= \{ T_{\bar e} \in \mathcal T \;\vert\; e_T \in D,\text{ and $T_{\bar e}$ has even order at least $2$}\}\\
\mathcal T^D_0   &= \{ T_{\bar e} \in \mathcal T \;\vert\; e_T \in D,\text{ and $T_{\bar e} \cong T_0 $ is trivial}\}
\end{align*}
We have defined the following partition of the general threads of $G$ into eight classes.
\begin{equation}\label{eq:allThreads}
    \mathcal T =  \mathcal T^1 \cup \mathcal T^\circ_{\textrm{odd}} \cup T^\circ_{\textrm{even}}
           \cup \mathcal T_0^M \cup \mathcal T_{\ge 1}^M \cup \mathcal T^D_{\textrm{odd}} \cup  \mathcal T^D_{\textrm{even}} \cup   \mathcal T^D_0 .
\end{equation}
By the choice of $M$, every injured thread in $\mathcal T$ belongs to $\mathcal T_{\ge 1}^M \cup \mathcal T^1$.

Let $\vec D$ be any fixed cyclic orientation of the $2$-factor $D = \bar G -M$.
For each general thread in~\eqref{eq:allThreads} we arbitrarily designate one of two possible flags to be its head,
subject to the following condition.
\begin{equation}\label{eq:orientation}
	\text{For every $T = T_{\bar e} \in \mathcal T^D_{\textrm{odd}}$, the head of $T$ equals the head of 
	$\bar e \in \vec D$.}
\end{equation}
We now refer to the thread weightings defined in Section~\ref{sec:threads}.
For every subset $\mathcal S \subseteq \mathcal T^D_{\textrm{odd}}$,
we define
$w_{\mathcal S} : E(G) \to \{0,1,2,3\}$ 
to be the edge weighting which restricts to every general thread $T \in \mathcal T$, as follows.
\begin{equation}\label{eq:wDef}
w_{\mathcal S} \upharpoonright_{E(T)} 
     =\begin{cases}
               w_{11}                   & \text{if  $T \in \mathcal T_{\ge 1}^M \cup \mathcal T^1$}\\
               w_{20}                                 & \text{if $T\in \mathcal S$},\\
               w_{02}                                & \text{if $T \in \mathcal (\mathcal T^D_{\textrm{odd}} - \mathcal S) \cup T^\circ_{\text {even}}$}\\
               w_{\mathbf 2}                   & \text{if $T \in \mathcal T_0^M \cup \mathcal T_0^D \cup \mathcal T^D_{\text{even}} \cup \mathcal T^\circ_{\text{odd}}$.}       \end{cases} 
\end{equation}
Finally, we define the following set of edge weightings of $G$,
\[
\mathcal W = \{w_\mathcal S : \mathcal S \subseteq \mathcal T^D_{\textrm{odd}}\}.
\]


\section{Star labellings of $G$}    \label{sec:starLabellings}
Let $G$  be a planar cubic graph.
We designate a proper block of $G$ to be the root block of~$G$.
We define the extended blocks, and $\bar G$, $M$, $\vec D$ and
$\mathcal W = \{w_\mathcal S : \mathcal S \subseteq \mathcal T^D_{\textrm{odd}}\}$ as in Section~\ref{sec:weightSet}.
Let $\Pi_\mathcal S$ be the set of star labellings of $G$ whose exponent is $w_\mathcal S$,
and let 
\[
\Pi = \cup \{\Pi_\mathcal S : \mathcal S \subseteq \mathcal T^D_{\textrm{odd}}\} .
\]
A \emph{base flag} of $G$ is any flag that is a foot of some general thread in $G$.
Thus every  base flag takes the form $v_He_H$
where $v_H$ is the unique vertex of degree 2 in some nonroot extended block $H^+$, and $e_H$ is the unique cut-edge of $G$ which is incident to $v_H$, and is not an edge of $H^+$.
\begin{proposition}
\label{prop:1-footed}
For every star labelling $\pi \in \Pi$ we have $\pi(v_H e_H)=1$, for every base flag $v_H e_H$ of $G$. 
\end{proposition}
\begin{proof}
Let $w_\mathcal S \in \mathcal W$ be the exponent of $\pi$.
Let $v_H e_H$ be a base flag in $G$,
and let $\mathcal K$ be the set of blocks $K$ of $G$  for which $e_H$ lies on the unique path from $K$ to the root block of $G$ in the block tree of $G$.
Let $L = \bigcup_{K \in \mathcal K} K^+$.
Every vertex in the subgraph $L$ has degree $3$ except for $v_H$, which has degree $2$.
For every extended block $K^+$ of $G$, the average value of $w_\mathcal S(e)$ among the edges $e \in E(K^+)$ equals~$2$.
This is because each of the weightings $w_{11}, w_{20},w_{02},w_{\mathbf 2}$  has average value $2$ in the definition of $w_\mathcal S$.
Since  $\{E(K^+): K \in \mathcal K\}$ is a partition of $E(L)$,
the average value of $w_{\mathcal S}(e)$ among the edges of $L$ equals $2$.
Therefore the average
value of $\pi(ve)$ among the flags in $F(L)$ equals $1$.
On the other hand, 
each vertex of $L$ is incident to three flags in $F(L) \cup \{ v_H e_H\}$ and contributes $0+1+2=3$ to the value of
$\pi(v_H e_H) + \sum \{ \pi(ve) : ve \in F(L)\}$.
Thus the average
value of $\pi(ve)$ among the flags in $F(L) \cup \{ v_H e_H\}$ equals $1$, and $\pi(v_H e_H)=1$.
\end{proof}

\begin{corollary}
\label{cor:restrictedSL}
Let $\mathcal S \subseteq \mathcal T^D_{\textrm{odd}}$.
Then every star labelling $\pi \in \Pi_\mathcal S$, satisfies the following.
	\begin{enumerate}
	\item For every $T \in \mathcal T_{\ge 1}^M \cup \mathcal T^1$ we have $\pi \upharpoonright_{F(T)} = \pi_{11}$.   \label{choice:11}

	\item For every $T\in \mathcal S$, we have $\pi \upharpoonright_{F(T)} = \pi_{20}$.                          \label{choice:20}

	\item For every $T\in \mathcal T^D_{\textrm{odd}} - \mathcal S$,           
	         we have $\pi \upharpoonright_{F(T)}\in \{\pi_{02}, \pi'_{11}\}$.                                        \label{choice:02}
	         
	\item For every $T\in  \mathcal T^\circ_{\text {even}}$,           
	         we have $\pi \upharpoonright_{F(T)}=\pi_{02}$.                                                          \label{choice:02circ} 

	\item For every $T\in \mathcal T^D_{\textrm{even}} \cup \mathcal T^\circ_{\text{odd}}$, 
	           we have $\pi \upharpoonright_{F(T)}  \in \{ \rho_{02} ,  \rho_{20}\}$.                       \label{choice:2}
	           
	\item For every $T\in \mathcal T^M_0 \cup \mathcal T^D_0$, 
	           we have $\pi \upharpoonright_{F(T)}  \in \{ \rho_{02} ,  \rho_{20}, \rho_{11}\}$.   \label{choice:0}
	\end{enumerate}
\end{corollary}
\begin{proof}
For any general thread $T \in \mathcal T$,
the restriction $\pi \upharpoonright_{F(T)}$ is 1-footed, by Proposition~\ref{prop:1-footed}.
Now all the statements except \eqref{choice:02circ} follow immediately
from the definition of $w_\mathcal S$ and
the three lemmas in Section~\ref{sec:threads}.
For the statement \eqref{choice:02circ}, we observe that the head and tail of a circular thread in
must receive distinct labels in any star labelling of $G$, and the claim follows from Lemma \ref{lem:uniq-circ}.
\end{proof}

For every star labelling $\pi$ of $G$, we define a a corresponding star labelling
$\bar \pi$ of $\bar G$ called the \emph{derived} star labelling.
Informally, $\bar\pi$ is the restriction of $\pi$  to the heads and the tails of the general threads of $G$.
More precisely, for each $\bar e \in E(\bar G)$, let $T = T_{\bar e}$ be the corresponding general thread of $G$.
Let $u$ and $v$ be the endpoints of $\bar e$ that correspond to the head and tail of $T$, respectively.
The restriction $\pi \upharpoonright_{F(T)}$ corresponds a 1-footed prestar labelling of a thread or injured thread having type $(i,j)$, for some $(i,j) \in \{(1,1),(2,0),(0,2)\}$.  We define $\bar\pi(ue)=i$ and $\bar\pi(ve)=j$.

Let $\pi \in \Pi$. By the definition of $\mathcal W$, the exponent of $\bar\pi$
is the constant function $\bar w = \mathbf 2$, $\bar w : E(\bar G) \to \{2\}$.
Let $M_\pi$ be the set of edges  $\bar e=uv\in E(\bar G)$ such that $\bar\pi(ue)=\bar\pi(ve)=1$.
Let $D_\pi = E(\bar G)-M_\pi$ and let $\vec D_\pi$ be the orientation of $D_\pi$
 where, for $e=uv \in \vec D_\pi$ we have $\bar\pi(ue)=0$ and $\bar\pi(ve)=2$.
Then $M_\pi$ is a perfect matching of $\bar G$, and $\vec D_\pi$ is an oriented $2$-factor of $\bar G$.
The correspondence between the star labellings of $\bar G$ with exponent $\mathbf 2$
and the pairs $(M',\vec D')$ where $\vec D'$ is an oriented $2$-factor of $\bar G$ is bijective.

The following will help us later deal with the exceptional star labelling $\pi'_{11}$ that arises in part  \eqref{choice:02} of Corollary \ref{cor:restrictedSL}.

\begin{lemma}   \label{lem:pi11OddCycle}
Let $\mathcal S \subseteq \mathcal T^D_{\textrm{odd}}$ and let $\pi \in \Pi_\mathcal S$.
Let $D_\pi$ be the $2$-factor of $\bar G$ as defined above.
If $G$ has a general thread $T \in \mathcal T$
for which
$\pi \upharpoonright_{F(T)} =\pi'_{11}$, then some connected component of $D_\pi$ is an odd cycle in $\bar G$.
\end{lemma}
\begin{proof}
Let $T$ be as in the statement. By Corollary \ref{cor:restrictedSL} we necessarily have $T\in \mathcal T^D_\text{odd} -\mathcal S$
Let $\mathcal T^\pi_{\ge 1}$ be the set of general  threads in $\mathcal T^M \cup \mathcal T^D$ which are nontrivial and receive a type $(1,1)$ prestar labelling under $\pi$.
By part \eqref{choice:11} of Corollary \ref{cor:restrictedSL}, we have $\mathcal T^M_{\ge 1} \subseteq \mathcal T^\pi_{\ge 1}$.
By the hypothesis, we also have  $T \in \mathcal T^\pi_{\ge 1} \setminus \mathcal T^M_{\ge 1}$ so $|\mathcal T^\pi_{\ge 1}| > | \mathcal T^M_{\ge 1}|$.
If the $2$-factor $D_\pi$ of $\bar G$ were bipartite, then the perfect matching $M_\pi$ would contradict our choice of $M$.
Therefore some component of $D_\pi$ is an odd cycle of $\bar G$.
\end{proof}

Let $(M, \vec D)$ be the perfect matching in $\bar G$ and the orientation of the complementary 2-factor used in the definition of the set of weightings $\mathcal W$.
Let $\pi^0$ be the unique star labelling of $G$ satisfying
	\begin{itemize}
	\item  $\pi^0 \in \Pi_\emptyset$,
	\item $M_{\pi^0}= M$, 
	\item every closed thread $T \in \mathcal T^\circ_{\text{odd}}$ receives the labelling $\rho_{02}$,
		as in \eqref{choice:2} of Corollary ~\ref{cor:restrictedSL}.
	\end{itemize}
The star labelling $\pi^0$ is called the \emph{reference star labelling} of $G$.
Accordingly, the reduced star labelling $\overline{\pi^0}$ is called the \emph{reference star labelling} of $\bar G$.
These star labellings will be used for sign computations. 
It convenient to assume that $\sgn(\pi^0)=\sgn\left(\overline{\pi^0}\right)=1$.

Let $\Pi_0$ be the set of star labellings $\pi\in\Pi$ for which $D_\pi$ is a bipartate $2$-factor of $\bar G$.
In particular, $\pi^0 \in \Pi_0$ because $D_{\pi^0}=D$ is bipartite.
The following was proved by Ellingham and Goddyn \cite{ElGo}.
\begin{lemma}
\label{EG}
Let $\bar\pi$ be a star labelling of a planar cubic graph $\bar G$ with exponent $w\equiv 2$.
If $D_\pi$ is bipartite, then $\sgn(\bar\pi)=1$.
\end{lemma}
\begin{corollary}
\label{cor:sgn}
Let $\pi \in \Pi_0$.
Then $\sgn(\pi) = (-1)^t$, where $t$ is the number of threads $T\in \mathcal T^D$ 
for which $\pi \upharpoonright_{F(T)}=\pi'_{11}$.
\end{corollary}
\begin{proof}
We have $\sgn(\pi) = \sgn(\bar\pi) \cdot \prod \left\{ \, \sgn\left(\pi \upharpoonright _{F(T)}\right) \;\vert\; T \in \mathcal T \,\right\}$.
The set of prestar labellings of $T \in \mathcal T$ whose exponents are determined by a weighting in $\mathcal W$ are restricted according to statements (1) to (6) of Corollary \ref{cor:restrictedSL}.
By Proposition \ref{prop:threadSigns}, 
both prestar labellings listed in statement \eqref{choice:2} have the same sign,
whereas the two labellings in statement \eqref{choice:02} have opposite sign.
In statements \eqref{choice:11}, \eqref{choice:20} and \eqref{choice:02circ}, the prestar labelling is fixed,
and in \eqref{choice:0} the sign is $1$ since the threads there are trivial.
By Lemma~\ref{EG}, $\sgn(\bar\pi)=1$ for $\pi \in \Pi_0$.
The result follows from the facts that  $t=0$ for $\pi=\pi^0$, and that $\sgn(\pi^0)=1$.
\end{proof}

Let $\Pi_1 = \Pi- \Pi_0$.
We now define a particular function $f $ which maps each member of $ \Pi_1$  to another star labelling of $G$.
We fix an arbitrary total ordering of the set of odd cycles in $\bar G$.
For $\pi \in \Pi_1$, let $C$ be the first odd cycle which is a component of $\bar G[D_\pi]$.
Every general thread  $T \in \{T_{\bar e} \;\vert\; e \in D_\pi\}$
has type $(0,2)$ or type $(2,0)$,
so we have $\pi \upharpoonright_{F(T)} \in \{\pi_{20}, \pi_{02}, \rho_{20}, \rho_{02}\}$,
and one of the cases \eqref{choice:20}, \eqref{choice:02}, \eqref{choice:2} or \eqref{choice:0} of Corollary \ref{cor:restrictedSL} applies to $T$.
(We used the fact $\pi'_{11}$ has type $(1,1)$.)
Let $f(\pi)$ be the star labelling in $\Pi_1$ obtained from $\pi$ as follows.
For every $\bar e \in E(C)$ we do the following.
If $T = T_{\bar e}$ is trivial, then we interchange the labels $0$ and $2$ on its two flags.
Otherwise, we relabel the flags of $T$ in a way that
interchanges either
the prestar labellings $\pi_{20}$ and $\pi_{02}$ (if $T\in \mathcal T^D_{\textrm{odd}}$),
or the prestar labellings $\rho_{20}$ and $ \rho_{02}$ (if $T\in \mathcal T^D_{\textrm{even}}$).
More precisely, for $\{i,j\}=\{0,2\}$, if $\pi \upharpoonright_{F(T)} = \pi_{ij}$, then $f(\pi) \upharpoonright_{F(T)} = \pi_{ji}$,
 and if $\pi \upharpoonright_{F(T)} = \rho_{ij}$, then $f(\pi) \upharpoonright_{F(T)} = \rho_{ji}$.

\begin{proposition}
\label{prop:involution}
The map $f$ is a fixed-point free involution $f : \Pi_1\to \Pi_1$ which satisfies  $\sgn(f(\pi)) = - \sgn(\pi)$.
\end{proposition}
\begin{proof}
Let $\pi \in \Pi_1$ and let $w_{\mathcal S}$ be the exponent of $\pi$.
Then  the exponent of $f(\pi)$ is the weighting $w_{\mathcal S'} \in \mathcal W$ where $\mathcal S' $ is the symmetric difference
of $\mathcal S$ and $\{\bar e \in E(C) \;\vert\; T_{\bar e} \in \mathcal T^D_\text{odd}\}$.
Therefore we have $f(\pi) \in \Pi_1$.
Clearly $f(\pi) \ne \pi$ and $f(f(\pi))=\pi$ so $f$ is a fixed-point free involution on  $\Pi_1$.
For $v\in V(C)$, the star labellings $\pi_v$ and $f(\pi)_v$ differ by the transposition $(02)$,
whereas $\pi_v = f(\pi)_v$ for every $v\in V(\bar G)-V(C)$.
Since $C$ has odd length, the derived star labellings therefore satisfy  $\sgn(\overline {f(\pi)}) = - \sgn(\overline \pi)$.
The result now follows from Corollaries~\ref{cor:sgn} and~\ref{prop:threadSigns}.
\end{proof}
We note that the oriented $2$-factor $\vec D_{f(\pi)}$ is obtained from $\vec D_\pi$ by reversing all the arcs in
the odd cycle $C$.

\section{The Main Theorem}

\begin{proof}[Proof of Theorem \ref{thm:main}]
Let the set of edge weights $\mathcal W$ be defined as in Section \ref{sec:weightSet},
and let $\Pi = \Pi_0 \cup \Pi_1$ be the star labellings of $G$ with exponent in $\mathcal W$, as defined in Section \ref{sec:starLabellings}.
Let $\pi \in \Pi_0$.
Then $D_\pi$ is a bipartite $2$-factor of $\bar G$.
Applying Lemma \ref{lem:pi11OddCycle}, we conclude that no general thread $T$ satisfies
$\pi \upharpoonright_{F(T)}=\pi'_{11}$.
It follows from Corollary \ref{cor:sgn} that $\sgn(\pi)=1$ for every $\pi \in \Pi_0$.
We have that $\Pi \ne \emptyset$, since $\Pi$ contains the reference star labelling~$\pi^0$.
Therefore $\sum_{\pi \in \Pi_0} \sgn(\pi) >0$.
We have by Proposition~\ref{prop:involution} that $\sum_{\pi \in \Pi_1} \sgn(\pi) =0$.
Thus we have shown that $\sum_{\pi \in \Pi} \sgn(\pi) > 0$.

Applying Corollary \ref{cor:interp} we have that $G$ is $(w+\mathbf 1)$-edge choosable for some $w=w_{\mathcal S} \in \mathcal W$.
We are interested in the upper bound $s(G,3) \le  |w^{-1}(3)| $. Suppose $G$ has $b$ cut-edges. 
For any general thread $T \in \mathcal T$, let 
\[
m(T) = \begin{cases}
                m        &   \text{if $T \cong T_m$} \\
                m+1    &   \text{if $T \cong T^\circ_m$} \\
                3          &  \text{if $T = H^+$ and $H$ is a vertex block of $G$}.
             \end{cases}
\]
Let $(i,j) \in \{(1,1),(2,0),(0,2)\}$. 
Define $\mathcal T_{ij} = \{ T \in \mathcal T \;\vert\;  w \upharpoonright_{E(T)} = w_{ij}\}$, 
let $n_{ij} = |\mathcal T_{ij}|$, and 
let $m_{ij} = \sum \{m(T)\;\vert\; T \in \mathcal T_{ij} \}$.  
Since every thread in $\mathcal T_{ij}$ has positive length we have $m_{ij} \ge n_{ij}$.  
Let $e$ be a cut-edge of $G$.  There are exactly two general threads  $T, T' \in \mathcal T_{11} \cup \mathcal T_{02} \cup \mathcal T_{20}$
such that $e$ joins a vertex of degree $\ge 2$ in $T$ to a vertex of degree $\ge2$ in $T'$.
Therefore each cut-edge $e$ contributes exactly twice to the quantity $m_{11}+m_{02}+m_{20}$, so
\begin{equation} \label{eq:nmBound}
n_{11} + n_{02} + n_{20} \le m_{11} + m_{02} + m_{20} =2b.
\end{equation}
Furthermore, at least one of the two contributions of $e$ goes toward $m_{11}$,
because the thread in $\{T, T'\}$ that lies farther from the root block $H_0$ is always a member of $\mathcal T_{11}$.
Therefore $m_{11} \ge m_{02}+ m_{20}$.
 
By examining  Figure~\ref{fig:threadWeights} we find that 
\begin{equation} \label{eq:nBound}
|w^{-1}(3)| = n_{11} + n_{02} + 2n_{20}.
\end{equation}
By comparing \eqref{eq:nmBound} and \eqref{eq:nBound},
we deduce that $s(G,3) \le 2n_{11} + 2n_{02} + 2n_{20} \le 4b$.
To obtain the claimed upper bound of $\frac52b$, we must argue more carefully. 

Let $w'$ be the edge weighting of $G$ obtained from $w=w_\mathcal S$ by interchanging the head and tail of every thread $T \in \mathcal T^D_\text{odd}$
(see \eqref{eq:orientation} in Section \ref{sec:weightSet}),
 and then swapping the roles of ``$\mathcal S$'' and ``$(\mathcal T^D_\text{odd} - \mathcal S)$''
in the definition \eqref{eq:wDef} of $w_\mathcal S$.
More precisely, the weightings $w$ and $w'$ are identical, except that for every thread $T \in \mathcal T^D_\text{odd}$, exactly one of the restricted weightings in
$\{w \upharpoonright_{E(T)}, w' \upharpoonright_{E(T)}\}$ coincides with $w_{02}$, and the other coincides with $w_{20}$ after exchanging the head and tail of $T$.
Then the coefficients of $x^w$ and $x^{w'}$ in $\epsilon(G)$ are equal in absolute value.  This is because there is a natural bijection from
the star labellings of $G$ with exponent $w$ to those with exponent $w'$.  
In the bijection, each star labelling $\pi$ with exponent $w$ maps to the unique star labelling $\pi'$ with exponent $w'$ which is identical on all flags outside of any thread in $\mathcal T^D_\text{odd}$, and for which the reduced star labellings of $\bar G$
satisfy 
$\bar \pi =\overline {\pi'}$. 

We now compute $s(G,3) \le \min (|w^{-1}(3)|, |(w')^{-1}(3)|) \le \frac12 (|w^{-1}(3)|+ |(w')^{-1}(3)|)$.
The analogue of \eqref{eq:nBound} is that $|(w')^{-1}(3)| = n_{11} + 2n_{02} + n_{20}$. We sum these two equations.
\[
2 s(G,3) \le 2n_{11} + 3n_{02} + 3n_{20} 
\le 4b +  (n_{02} + n_{20}) .
\]
We apply the inequality just before \eqref{eq:nBound}.
\[
2 (n_{02} + n_{20}) \le 2 (m_{02} + m_{20}) \le m_{11}+m_{02} + m_{20} = 2b.
\]
So $2 s(G,3) \le 5b$.  That is to say, that one of the two functions, $f = w + \mathbf 1$ or $f = w' + \mathbf 1$, satisfies the statement of Theorem \ref{thm:main}.
\end{proof}

\section{Comments}
Our use of the weightings $w_{\mathbf 2}$ in the definition \eqref{eq:wDef} of $w_\mathcal S$
 was included as a mild attempt to minimize the number of
edges of $G$ receiving weight $3$. 
In particular, extended cycle blocks of even length, (open)  threads of even order, and edges in trivial threads do not require lists of length $4$.
By inspecting the above inequalities, it is apparent that the upper bound $s(G,3)\le \frac52 b$ can be improved 
if many threads of $G$ have length $>1$, or if $\bar G$ has a proper $3$-edge colouring in which most of the edges coming from nontrivial threads have the same colour.
It would be very interesting to improve the bounds  $2b \le s(\text{planar cubic}, 3) \le\frac52b$,
where $s(\text{planar cubic},3) = \sup\{ \:s(G,3) \;\vert\; G \text{ is planar and cubic} \}$.
Similar questions can be asked regarding regular planar graphs of higher degree.


\begin{thebibliography}{99}

\bibitem{Alo} N. Alon, Restricted colorings of graphs,
{\it in\/} ``Surveys in Combinatorics'', Proc. $14^{th}$ British
Combinatorial Conference,  London Mathematical Society Lecture
Notes Series 187, edited by K. Walker, Cambridge University Press,
1993, 1-33.

\bibitem{AloTar} N. Alon and M. Tarsi, Colorings and orientations
of graphs, {\it Combinatorica\/} {\bf 12} (1992), 125--134.

\bibitem{AH} Appel, K.; Haken, W.; Koch, J. Every planar map is four colorable. II. Reducibility.  
\textit{Illinois J.\ Math.}\  21  (1977),  491--567

\bibitem{ElGo} M.\ Ellingham, L.\ Goddyn, List edge colourings of some regular planar multigraphs, 
\textit{Combinatorica} 16 (1996), 343--352.

\bibitem{ERT}  Erd\"os, Paul; Rubin, Arthur L.; Taylor, Herbert,  Choosability in graphs. 
\textit{Combinatorics, graph theory and computing}, Proc.\ West Coast Conf., Arcata/Calif.\ 1979, 125--157 (1980).

\bibitem{Jaeger} Jaeger, Francois, On the Penrose number of cubic diagrams, 
{\it Discrete Math.\/} {\bf 74} (1989) 85--97.

\bibitem{Galvin} Galvin, Fred, The list chromatic index of a bipartite multigraph. 
\textit{J.\ Combin.\ Theory Ser.\ B} 63 (1995), 153--158. 

\end{thebibliography}
\end{document}